\documentclass[11pt]{amsart}
\usepackage{amsmath,amssymb,amsthm,amscd}
\title{On the tautological ring of a Jacobian modulo rational equivalence}
\author{Baohua Fu and Fabien Herbaut}

\newtheorem{Thm}{Theorem}
\newtheorem{Lem}{Lemma}
\newtheorem{Prop}{Proposition}

\def\qit{{\mathbb Q}}
\def\zit{{\mathbb Z}}
\def\pit{{\mathbb P}}

\def\0{{\mathcal O}}

\def\I{{\mathcal I}}

\def\D{{\mathcal D}}
\def\F{{\mathcal F}}
\def\G{{\mathcal G}}

\def\T{{\mathcal T}}

\def\taut{{\mathcal{T}}}

\begin{document}
\maketitle
\begin{abstract}
We consider the Chow ring with rational coefficients of the
Jacobian of a curve. Assume $D$ is a divisor in a base point free
$g^r_d$ of the curve such that the canonical divisor $K$ is a
multiple of the divisor $D$. We find relations between
tautological cycles. We give applications for curves having a
degree $d$ covering of $\mathbb{P}^1$
 whose ramification points are all of order $d$, and then for hyperelliptic curves.
\end{abstract}
\section{Introduction}
\subsection{} For an abelian variety $X$ of dimension
$g$, we denote by $CH(X) = \oplus_{p=0}^g CH^p(X)$ the Chow ring
of $X$ with {\em rational coefficients}. Let $k: X \to X$ be the
morphism $x \mapsto kx$ for any $k \in \zit$. It turns out that
all the pull-backs $k^*$ and push-forwards $k_*$ on the level of
Chow rings can be diagonalized simultaneously, which gives the
following decomposition (\cite{B2}):
\begin{equation}
\label{decomp_beauville} CH^p(X)=\bigoplus_{i=p-g}^{p}
CH^p_{(i)}(X),
\end{equation}
where $CH^p_{(i)}(X)$ consists of elements $\alpha \in CH^p(X)$ such
that $k^* \alpha = k^{2p-i} \alpha$ (which is equivalent to  $k_*
\alpha=k^{2g-2p+i} \alpha$) for all $k \in \mathbb{Z}$.

\subsection{}
When $(X, \theta)$ is a principally polarized abelian variety, we
may identify $X$ with its dual. Let $p_i: X \times X \to X, i=1,
2$ be the two projections and $m: X \times X \to X$ the addition
map. The Poincar\'e line bundle on $X \times X$ is represented by
the divisor $P: = p_1^* \theta + p_2^* \theta - m^* \theta$. The
Fourier transform $\F: CH(X) \to CH(X)$ is defined by $\F(\alpha)
= (p_2)_*(p_1^* \alpha \cdot e^P)$.  The morphism $\F$ turns out
to be an automorphism of the $\qit$-vector space $CH(X)$. We refer
to \cite{B1} and \cite{B2} for a detailed study of the fundamental
properties of $\F$.

\subsection{}
The particular case of the Jacobian $J$ of a smooth projective
curve $C$ of genus $g \geq 2$ has been studied by several authors. For
any point $x_0 \in C$, we have a natural map: $\iota: C \to J$
given by $x \mapsto \0_C(x-x_0).$ We denote by $\T$ the smallest
subring of $CH(J)$ containing the class $[\iota(C)] = \iota_*[C]$
which is closed under the pull-backs $k^*$ and push-forwards $k_*$
for all $k \in \zit$ and under the Fourier transform $\F$, which
will be called the tautological subring of $CH(J)$. Notice that
this ring depends on the choice of the base point $x_0 \in C$.

Let us consider the decomposition of $[\iota(C)]$ given by
 \eqref{decomp_beauville} (note $CH^{g-1}_{(-1)}(J) = 0$):
\begin{equation}\label{decomp_courbe} [\iota(C)]=C_{(0)} + C_{(1)} +
\ldots + C_{(g-1)} \ \ \textrm{where } \ \ C_{(i)} \in
CH^{g-1}_{(i)}(J).
 \end{equation}
Let $\theta$ be a symmetric theta divisor on $J$.  As in \cite{P2}, we
define
\[ p_i=\F(C_{(i-1)}) \in CH^i_{(i-1)}(J) \textrm{ for } i \geq 1,
\
 q_i=\F(C_{(i)} \cdot \theta) \in CH^i_{(i)}(J) \textrm{ for } i \geq 0 \textrm{.}\]

Let $\simeq^{alg}$ be the algebraic equivalence. Note that $q_i
\simeq^{alg} 0$ for $i\geq 1$. Beauville proved in \cite{B3} that
the ring $\T/\simeq^{alg}$ is generated (as a subring under the
intersection product) by the classes $(p_i)_{i \geq 1}$. Recently,
Polishchuk proved in \cite{P2} (Thm 0.2 and Prop. 4.2) that $\T$
is generated by the classes $(p_i)_{i<g/2+1}$ and $(q_i)_{i <
(g+1)/2}$.

\subsection{}

The influence on $\T/\simeq^{alg}$  of the gonality of the curve
$C$ has been firstly studied by Colombo and van Geemen in
\cite{CvG}. They proved that if $C$ admits a $g_d^1$, then $p_i$
is algebraically equivalent to zero if $i \geq d$. Recently the
second named author computed some relations between the cycles
$p_i$ in $\T/\simeq^{alg}$ for a curve admitting a base point free
$g_d^r$. In \cite{GK}, van der Geer and Kouvidakis gave another
proof of the main result in \cite{Her}, by using the
Grothendieck-Riemann-Roch theorem. They gave simpler but equivalent relations
as Don Zagier proved it. Note that Polishchuk has found
universal relations in $\T / \simeq^{alg}$ (cf \cite{P1}) and in
$\T$ (cf \cite{P2}).

We will denote by $A(r, d, g)$
the Castelnuovo number, which is:
$$
A(r,d,g) = \sum_{i=0}^{r-1} \cfrac{(-1)^i}{d-2r+2}
\binom{i+g+r-d-2}{i} \binom{d-2r}{r-1-i} \binom{d-r+1-i}{r-i}.
$$
It was explained in \cite{Her} that if $A(r,d,g) \neq 0$ one can deduce
 from Colombo and van Geemen's Theorem that $p_i \simeq^{alg} 0$
 when $i \geq d - 2r+2$. In particular,
the ring $\T/\simeq^{alg}$ is generated by $(p_i)_{i \leq
d-2r+1}$.

\subsection{}
Our aim in this note is to generalize the above results to $\T$.
More precisely we shall prove the following:

\begin{Thm} \label{main}
Let $C$ be a smooth projective curve which admits a base
point free $g_d^r$ with $D$ being a class in $g_d^r$. Let $K$ be
the canonical class of $C$. If  $K$ is a multiple of $D$ in
$CH^1(C)_\qit$ and if $A(r,d,g) \neq 0$, then the tautological ring is generated by $(p_i)$
and $(q_i)$ with $i \leq d-2r+1$.
\end{Thm}

We present two proofs : the first follows the method  in \cite{GK} and
uses an argument in \cite{Her}. The second one follows the methods of \cite{Her}.
In the two proofs we use the operator defined in \cite{P2}.
In both approaches, we encounter the same difficulty which
can be overcomed by assuming that $K$ is a multiple of $D$.
This condition is satisfied for example for curves which are complete
intersections. In this case our result is much sharper than the aforementioned
result of Polishchuk.

Some applications and refinements of this theorem are given in the
last sections, where we show (see Thm. \ref{pure}) that if $C$ is
a curve admitting a  degree $d$ map $C \to \pit^1$ whose
ramification points are all of order $d$, then the tautological
ring is generated by $(p_i)_{i \leq d-1}$ and $q_1$. In the case
of a hyperelliptic curve, we obtain a precise description of $\T$
(Prop. \ref{hyper}), which generalizes a result of Collino
(\cite{Co}, Thm. 2) on the validness of Poincar\'e's formula under
rational equivalence.
\section{Preliminary results}
According to Cor. 24 in \cite{Ma} or to Thm. 4 in \cite{Her}, if
$p_i \simeq^{alg} 0$ then $p_{j} \simeq^{alg}  0$ for $j \geq i$.
We generalize this to the case of rational equivalence in this
section, which will be needed in the proof of Theorem \ref{main}.
Recall the following operator $\D$ introduced in \cite{P2}:
$$
\begin{aligned}
\mathcal{D} = & \frac{1}{2}\sum_{m,n \geq 1} \binom{m+n}{n}
p_{m+n-1}
\partial_{p_m} \partial_{p_n} \\  &+   \sum_{m,n \geq 1}
\binom{m+n-1}{n} q_{m+n-1} \partial_{q_{m}} \partial_{p_{n}}-
\sum_{n \geq 1} q_{n-1}\partial_{p_n}.
\end{aligned} $$

Let $\I_q \subset \qit[p,q]$ be the ideal generated by the
elements $(q_i)_{i \geq 1}$, so $\I_q \simeq^{alg} 0$. By the
expression of $\D$, we have $\D (\I_q) \subset \I_q$.
\begin{Prop}\label{lem1}
Let $n$ be an integer such that $2 \leq n \leq g-1$. Suppose that
$p_n \in \I_q$, then

 i) for any $m$ such that $n \leq m
\leq g$,  $p_m$ is contained in the ideal generated by $(q_i)_{1
\leq i \leq m-1}$ ;

 ii) the tautological ring $\T$ is generated by $(p_i)_{i
\leq n-1}$ and $(q_i)_{i \leq n-1}$.

\end{Prop}
\begin{proof}
To prove i) we use an induction on $m$. Assume $m\geq 2$ and $p_m \in \I_q$, then $p_2 p_m
\in \I_q$. We apply the operator $\D$, which gives $\D(p_2 p_m)
\in \I_q$.  Note that
$\mathcal{D}(p_2p_m)=\binom{m+2}{2}p_{m+1}-q_1p_m-q_{m-1}p_2$,
thus $p_{m+1} \in \I_q$. If we write $p_{m+1} = \sum_{i \geq 1} q_i A_i$
with $A_i \in \T$, then $A_i = 0$ for $i \geq m+1$, since $q_i \in
CH^i _{(i)}$, $p_{m+1} \in CH^{m+1}_{(m)}$ and $A_i \in \oplus_{j
\geq 0}CH_{(j)}$.

To prove ii), we just need to prove that for any $m$
such that $n \leq m \leq g$, the class $q_m$ is contained in the
ideal generated by $(q_i)_{i \leq n-1}$. Assume this is true for
all $m <k$, we shall prove it for $m=k$. By i) and
the hypothesis of induction, we have $p_k = \sum_{j=1}^{n-1} q_j
A_j$ with $A_j \in \mathcal{T}$. Now the claim follows from the following relation in
\cite{P2}:
\begin{displaymath}
q_k \ = \ \mathcal{D} \big{(} q_1p_k\big{)} + q_1q_{k-1} =
\sum_{j=1}^{n-1} \D \big{(} q_1 q_j A_j\big{)} + q_1q_{k-1} .
\end{displaymath}
\end{proof}

\section{First proof of the Theorem}
We shall follow the method of \cite{GK} coupled with arguments in
\cite{Her} to prove our main theorem.

Let $C$  be a smooth projective curve with a base-point-free
linear system $g_d^r$, which gives arise to a morphism: $\gamma: C
\to \pit^r$. Let $Y \subset C \times \check{\pit}^r$ be the
subvariety $\{(p, y)| \gamma(p) \in y  \}$. The projection $\phi:
Y \to C$ is a $\pit^{r-1}$-bundle. More precisely if we denote by
$E: = \gamma^* (T^* \pit^r)$, then $Y = \pit(E)$ as a projective
bundle over $C$. Recall that we have the following relative Euler
sequence:
$$ 0 \to \0_Y \to \phi^*(E) \otimes \0_\phi(1) \to T_Y \to \phi^* T_C \to 0.  $$
This gives $td(T_Y) = td(\phi^* T_C) \cdot td(\phi^*(E)\otimes
\0_\phi(1)) = (1-\frac{1}{2} \phi^* K) \cdot td(\phi^*(E)\otimes
\0_\phi(1)),$ where $K$ is a canonical divisor of the curve $C$.

Let us denote by $\tilde{\alpha}: Y \to \check{\pit}^r$ the second
projection, which is a finite morphism of degree $d$. Note that
$c_1(\phi^*(E)) = \phi^* (\gamma^*(c_1(T^* \pit^r))) = -(r+1)
\phi^* D$, where $D$ is a divisor in the linear system $g_d^r$.
Furthermore the higher Chern classes of $\phi^*(E)$ vanish, as it
is the pull-back from a vector bundle over a curve. This gives
that $td(\phi^*(E)\otimes \0_\phi(1))$ is a linear function in
$\phi^*(D)$ with coefficients in $\tilde{\alpha}^*
CH(\check{\pit}^r).$

From now on, we will identify $\pit^r$ with $\check{\pit}^r$ and
make use of the following self-explaining notations:
\begin{equation*}\begin{CD}
\pit^r @<v<< \pit^r \times J @<\alpha<<  Y \times J @>\pi>> Y @>\tilde{\alpha}>>  \pit^r \\
  @.           @VpVV                      @V{\psi}VV         @V{\phi}VV  @. \\
@.     J @<q<<         C\times J @>{\tilde{\pi}}>> C @.
\end{CD}
\end{equation*}

By the precedent discussions, the relative Todd class $td(\alpha)
= \pi^* td(\tilde{\alpha})$ is of the form $(1-\frac{1}{2} \pi^*
\phi^*K)(A(x)+B'(x) \pi^* \phi^* (D)) = A(x) + B'(x) \pi^* \phi^*D
- \frac{1}{2} A(x) \pi^* \phi^* K,$ for some polynomials $A, B'$
of degrees at most $r-1$, where $x = \alpha^*(\xi)$ and $\xi = v^*(h)$
for a hyperplane $h \subset \pit^r$.

Let us denote by $\Pi$ the divisor of the Poincar\'e line bundle
on $C \times J$ and $l: = \psi^* (\Pi)$. Let $L $ be the line
bundle on $Y \times J$ associated to $l$. As in [GK], we apply
the Grothendieck-Riemann-Roch
formula to $V_k: = \alpha_*(L^{\otimes k})$: $$ ch(V_k) =
\alpha_*(e^{k l} td(\alpha)) = \alpha_* ( e^{kl}A(x) + e^{kl}
B'(x) \pi^* \phi^*D - \frac{1}{2} e^{kl} A(x) \pi^* \phi^* K).$$

The following Lemma can be proved in a similar way as in
[GK](Lemma 3.3).
\begin{Lem}
In $CH(\pit^r \times J)$, we have the following relations for $\nu
\geq 0$:
\begin{equation*}
\alpha_*(l^\mu \cdot x^\nu) = \begin{cases} & q_*(\Pi^\mu)
\xi^{\nu +1} + q_*(\Pi^\mu \cdot \tilde{\pi}^*D) \xi^\nu, \text{if
$\mu > 0 $}, \\ &d \xi^\nu, \text{if $\mu=0$}.
\end{cases}
\end{equation*}
For any divisor $D_0$ on $C$, we have
\begin{equation*}
\alpha_*(l^\mu \cdot x^\nu \cdot \pi^* \phi^* D_0) = \begin{cases}
& q_*(\Pi^\mu \cdot \tilde{\pi}^* D_0) \xi^{\nu +1}, \text{if $\mu
> 0 $}, \\ &\deg(D_0) \xi^{\nu+1}, \text{if $\mu=0$}.
\end{cases}
\end{equation*}
\end{Lem}
Using these formulae, one can easily obtain
\begin{align*}
\alpha_*(e^{kl} \cdot A(x)) & = \alpha_*(A(x))
 +\sum_{\mu=1}^\infty \frac{k^\mu}{\mu !} ( q_*(\Pi^\mu) \xi
A(\xi) +  q_*(\Pi^\mu \cdot \tilde{\pi}^* D) A(\xi))) \\
&= d A(\xi) + q_*(e^{k\Pi}) \xi A(\xi) + q_*(e^{k \Pi} \cdot
\tilde{\pi}^* D) A(\xi) - q_*(\tilde{\pi}^* D) A(\xi) \\
&= q_*(e^{k\Pi}) \xi A(\xi) + q_*(e^{k \Pi} \cdot \tilde{\pi}^* D)
A(\xi).
\end{align*}

In a similar way, one obtains $\alpha_*(e^{kl} B'(x) \cdot \pi^*
\phi^* D) = q_*(e^{k\Pi} \cdot \tilde{\pi}^* D) \xi B'(\xi)$ and
$\alpha_*(e^{kl} A(x) \cdot \pi^* \phi^* K) = q_*(e^{k\Pi} \cdot
\tilde{\pi}^* K) \xi A(\xi). $


Recall that we have fixed a base point $x_0 \in C$ and considered the natural map $\iota:
C \to J$.
\begin{Lem}
For any $k \in \zit$ and any cycle $D' \in CH(C)_\qit$, we have
$\F(k_* \iota_* D') = q_*(e^{k\Pi} \cdot \tilde{\pi}^* D').$
\end{Lem}
\begin{proof}
We have a Cartesian diagram:
\begin{equation*}
\begin{CD}
C \times J  @>\iota'>>  J \times J \\ @V\tilde{\pi}VV  @Vp_1VV \\
C @>\iota>> J
\end{CD}
\end{equation*}

Let $q': J \times J \to J$ be the second projection and $P$ the
Poincar\'e divisor on $J \times J$. We have $\F(k_* \iota_* D') =
q'_*(p_1^* k_* \iota_* D' \cdot e^P) = q'_*(k'_* p_1^* \iota_* D'
\cdot e^P) = q'_* k'_* (p_1^* \iota_* D' \cdot e^{kP}) = q'_*
(\iota'_* \tilde{\pi}^* D' \cdot e^{k \Pi}) = q_*(e^{k\Pi} \cdot
\tilde{\pi}^* D').$ Here $k': J \times J \to J \times J$ is given
by $(x, y) \mapsto (kx, y)$.
\end{proof}

Finally we obtain the following:
\begin{equation*}
ch(V_k) = \F(k_*\iota_* [C]) \xi A(\xi) + \F(k_* \iota_*
D)(A(\xi)+\xi B'(\xi)) - \frac{1}{2} \F(k_* \iota_* K) \xi A(\xi).
\end{equation*}

Note that modulo algebraic equivalence, this gives exactly the
formula in Prop. 3.1 in [GK]. From now on, we assume furthermore
that $K = \frac{1}{s} D$ in $CH^1(C)_\qit$ for some $s \in \qit$.
Let $B(\xi) = sA(\xi) + s \xi B'(\xi) - \frac{1}{2} \xi A(\xi)$,
then we get a simpler formula for $ch(V_k)$ as follows:
$$
ch(V_k) = \F(k_* \iota_* [C]) \xi A(\xi) + \F(k_* \iota_* K)
B(\xi).
$$

Recall that $\F(k_* \iota_* [C])= \sum_{i=1}^{g} k^{i+1} p_i $ and
$\F(k_* \iota_* K) = (2g-2) + 2 \sum_{i=1}^g k^i q_i. $ Let us
write $\xi A(\xi) = \sum_{i=0}^{r-1} a_i \xi^{i+1}$ and $B(\xi) =
\sum_{i=0}^r b_i \xi^i$, then a direct calculus shows:
$$
ch_j(V_k) = 2sk^j q_j + \sum_{i=1}^{j-1}(k^{j-i+1}a_{i-1} p_{j-i}
+ 2 k^{j-i} b_i q_{j-i}) \xi^i + (2g-2) b_j \xi^j.
$$
To simplify the notations, we will put $ch_j(V_k) = \sum_{m=0}^j
A_m(j) \xi^m$, where $A_m(j)$ is of codimension $j-m$.

Note that modulo $\I_q$, the expressions of $A_m(j)$ are the same
as those in \cite{GK}. By the proof of Prop. 3.6 and that of the
main Theorem in \cite{GK} we obtain the following generalization
of the main result in \cite{GK}:
\begin{Thm}\label{second}
If $C$ admits a base point free $g_d^r$ with $D \in g_d^r$  and
$K$ is a multiple of $D$ in $CH^1(C)_\qit$, then $ \sum_{\alpha_1
+ \cdots +\alpha_r = M-2r+1}  (\alpha_1+1)! \cdots (\alpha_r+1)!
p_{\alpha_1+1} \cdots p_{\alpha_r+1}$ is contained in $\I_q$ for
all $M \geq d$.
\end{Thm}

When $r \geq 2$, we recall the Castelnuovo number:
$$
A(r,d,g) = \sum_{i=0}^{r-1} \cfrac{(-1)^i}{d-2r+2}
\binom{i+g+r-d-2}{i} \binom{d-2r}{r-1-i} \binom{d-r+1-i}{r-i}.
$$

Consider the case $M=d$ in the precedent theorem. Then the
argument in section 6.2 [Her] shows that $p_{d-2r+2}$ is contained
in $\I_q$ (here we note that Cor. 0.3 in [P1], which is used in
Section 6.2 \cite{Her}, holds in fact in $\T$ modulo $\I_q$
instead of $\T/\simeq^{alg}$). This concludes the proof of the
Theorem \ref{main}
 by Prop. \ref{lem1}. \\ \\
\textbf{Remark : } i) The condition in the theorem on the
relationship between $K$ and $D$ can be weakened to that
$\F(\iota_* D)-deg(D)[J]$ is contained in
$\I_q$.\\
ii) This method also gives a way to compute the terms in $\I_q$,
which in turn
gives relations in the tautological subring $\T$. \\ \\
\textbf{Examples :}
Let $C \subset \mathbb{P}^r$ be a smooth curve which is a
complete intersection of hypersurfaces of
degrees $(d_1, \ldots, d_{r-1})$. One shows easily that
 $K_C=\mathcal{O}_{\mathbb{P}^r}(\sum d_i -r-1) |_C $ which
is a multiple of $D := C \cap H \in g^r_d$ with $d= \prod d_i$
and $H$ is a hyperplane in $\mathbb{P}^r$. If $d_1, \ldots, d_{r-1}$
are sufficiently big, then $A(r,d,g) >0$ and Theorem \ref{main} applies.
We note that in this case $d-2r+1$ is much smaller than
$\frac{g}{2}+1$ because $g = \frac{( \sum d_i -(r+1) )d}{2}+1$, thus this
result is sharper than that in \cite{P2}. \\ \\

\section{Another proof of the Theorem}
In this section we give the sketch of a proof of Theorem \ref{main}
based upon the method described in \cite{Her}.
Although this approach is different, we encounter the same difficulties
as above, i.e we need a relation between $D$ and $K$ to proceed.

We will note $C_n$ for the $n$-th symmetric product of $C$. Suppose
that $C$ carries $\G$ a base point free $g^r_d$
 which contains $D=o_1 + \ldots + o_d$.
 Theorem 3 in \cite{Her} gives the class $\G$ in $CH(C_d)$ in terms
of the classes of the diagonals
$\delta_{i_1,\ldots,i_r}=\{ i_1x_1+\ldots+i_rx_r | (x_1,\ldots,x_r) \in C^r \}$ :
\begin{displaymath}
[\G] =   \sum_{ \scriptstyle 1 \leq i_1 \leq \ldots \leq i_r  \atop \scriptstyle 1 \leq j_1 < \ldots < j_s \leq d    } \Big{(}\prod_{u=1}^{r} \frac{(-1)^{i_u-1}}{i_u} \Big{)} \ \  [ \ \delta_{i_1,\ldots,i_r}+o_{j_1}+\ldots+o_{j_s} \ ],
\end{displaymath}
where $s=d-\sum_{j=1}^r i_j$.

Consider the morphism $u_d \ : \ C_d \rightarrow  J$ which
maps $D$ onto the class of the divisor $D-dx_0$.
It contracts $\G$ onto
a point in the Jacobian, so  ${u_d}_*[\G]=0$. Moreover,
${u_d}_* [ \delta_{i_1,\ldots,i_r} + o_1 + \ldots + o_{d-\sum i_a} ]$
is a multiple of $(i_1*C )* \ldots * (i_r*C )* [o_1+ \ldots +o_{d-\sum i_a}-(d-\sum i_a)x_0]$.
Then we use the same arguments as in the begining of section 5 in \cite{Her},
this is the decomposition (\ref{decomp_courbe}),
the relations $i_*C_{(a)}=i^{a+2}C_{(a)}$
and the  projection onto $CH(J)^{g-r}_{(s)}$ for $s \geq 0$.  We get :
\begin{equation}\label{relation2}
\sum_{\scriptstyle 0 \leq a_1,\ldots ,a_r,t \atop \scriptstyle a_1+\ldots + a_r+t=s }
  \ C_{(a_1)}* \ldots* C_{(a_r)}*\beta(d,a_1,\ldots,a_r)_{(t)} = 0  \textrm { \ \ where  }
\end{equation}
\[\beta(d,a_1,\ldots,a_r)=
\sum_{i_1=1}^{d} \ldots \sum_{i_r=1}^{d}
 (-1)^{i_1+\ldots +i_r }  {i_1}^{a_1+1} \ldots {i_r}^{a_r+1} \alpha_{d-i_1-\ldots -i_r} \textrm{ ,}\]
and where
$$\alpha_u= \sum_{\scriptstyle 1 \leq k_1<k_2< \ldots < k_u \leq d} \  [o_{k_1}+ \ldots + o_{k_u}-u \ x_0] .$$
We do not know if the cycles $\alpha_u$ are in $\T$ in general.
As in the previous section, we only know how to obtain
relations in $\T$ when $D$ is a multiple of $K$,
which we will suppose from now on.
In this case $\alpha_1$ is a multiple of $\iota_* K$ and
thus belongs to $ \taut$. Then we use an induction and
the relations
$$
(i+1) \alpha_{i+1} = \alpha_i * (\alpha_1) - \alpha_{i-1}*(2_*
\alpha_1)+ \cdots + (-1)^{i-1} \alpha_1*(i_* \alpha_1)+(-1)^i(i+1)_* \alpha_1
$$
to prove that $\alpha_u \in \taut$ for $0 \leq u \leq d$. Now let
us apply $\F$ to \eqref{relation2} and deduce relations
modulo $\I_q$. As the
elements $\F ( \alpha_u ) $ are Fourier transforms of $0$-cycles, they
belong to $\bigoplus_{s \geq 0} (\T \cap
CH(J)^{s}_{(s)})=\mathbb{Q}[q_0,q_1,\ldots,q_g]$ and we
have $ \F ( \alpha_u ) = \F ( \alpha_u )_{(0)} \  (  \textrm{mod } \I_q  )$.
 Then, $\F(\beta(d,a_1,\ldots,a_r)_{(t)}) \in \I_q$ if $t>0$.
As
$$ \F(\alpha_u)_{(0)} \ = \ \sum_{1 \leq k_1 < \ldots < k_u \leq d} [J] \ = \ \binom{d}{u} [J], $$
we obtain the relations
\[ \sum_{\scriptstyle 0 \leq a_1,\ldots ,a_r \atop \scriptstyle a_1+\ldots + a_r=s }
 \gamma(d,a_1,\ldots,a_r) \ p_{a_1+1}\ldots p_{a_r+1} \ = \ 0 \ \ \left ( \textrm{ mod } \I_q \right ) \ \ \textrm { \ \ where  }\]
\[ \gamma(d,a_1,\ldots,a_r)=
\sum_{i_1=1}^{d} \ldots \sum_{i_r=1}^{d} (-1)^{i_1+\ldots +i_r }
 \binom{d}{i_1+\ldots + i_r} {i_1}^{a_1+1} \ldots {i_r}^{a_r+1} \textrm{ .}\]
Now according to the appendix of Don Zagier in \cite{GK} these
relations are the same as those in Theorem \ref{second}. Then using Prop \ref{lem1}, a
similar argument as that in the precedent section concludes the
proof.

\section{Curves admitting a pure ramification}
We begin with the following easy result.
\begin{Prop}\label{prop2}
Let $C$ be a smooth projective curve of genus $g$ and  $x \in C$ a
point such that $K_C = (2g-2) x$ in $CH^1(C)_\qit$. Then
 the tautological ring $\T$ on $J$
    is generated over $\qit$ by the classes $(p_n)_{n < g/2+1}$ and
$q_1$.
\end{Prop}
\begin{proof}
Let $[c]$ be the image of the canonical divisor of $C$ in $J$ and
write $[c] = [o] + c_1+ \cdots +c_g$ the decomposition after
\eqref{decomp_beauville}. By the proof of Prop. 4.3 in [Pol], we
have $\F([c])= exp(2q_1)$, which gives $\F(c_k)=(2q_1)^k/k! =
\F(c_1)^k/k!$. As a consequence, we obtain a Poincar\'e-type
formula: $ c_k = \frac{c_1^{*k}}{k!}.$

Now suppose that there exists a point $x \in C$ such that $K_C =
(2g-2) x$ in $CH^1(C)_\qit$. Then $\iota_* K = 2(g-1)[x-x_0]$,
which gives that $(2g-2)_* \iota_* K = 2(g-1)[c].$ Let
$\eta=\iota_*K/2+[0]$ and $\eta_n$ the component of $\eta$ in
$CH^g_{(n)}$. By the formula on page 4 of \cite{P2}, we have $q_n
= \F(\eta_n)$. The precedent discussions give that
 $c_k = 2(2g-2)^{k-1}
\eta_k$.  This gives $2(2g-2)^{k-1} q_k = \F(c_k) =
\F(c_1^{*k})/k! = \F(c_1)^k/k! = (2^k/k!)q_1^k$, which gives
$$q_k = \cfrac{q_1^k}{k! (g-1)^{k-1}}.$$   By Prop. 4.2 of [Pol],
we obtain the claim.
\end{proof}
\noindent \textbf{Remark : }
Similar argument can be applied to more general cases. For
example, if $K=(g-1)(x+y)$ for some points $x, y \in C$. Then we
have $\iota_* K = (g-1)([x-x_0]+[y-x_0]) =
[x+y-2x_0]*(-1)_*\iota_* K$. Apply $(g-1)_*$ to both sides, we
obtain $(g-1)_* (2\eta-2[o]) = [c]*(1-g)_*(2\eta - 2[o])$. This
gives that $$ [o]+\sum_{i=1} (g-1)^{i-1} \eta_i= ([o]+\sum_{i=1}
c_i)* ([o] + \sum_{i=1} (-1)^i (g-1)^{i-1} \eta_i).
$$

The first non-trivial relation is $2(g-1)^2 \eta_3
= c_3-\eta_1*c_2+(g-1)\eta_2*c_1$.  Applying the Fourier
transformation, we obtain: $2(g-1)^2 q_3 = -2/3q_1^3 +
2(g-1)q_2q_1.$ Similarly we obtain expressions of $q_5, q_7
\cdots$.  This gives that the tautological ring is generated by
$(p_n)_{n <g/2+1}$, $q_1$ and $(q_{2k})_{2k <g/2+1}$. \\

We say that a curve $C$ admits a {\em pure ramification} of degree
$d$ if there exists a degree $d$ map $C \to \pit^1$ whose
ramification points are all of order $d$. In this case, if we take $x$
to be a ramification point, then by the Hurwitz's formula, we have
$K=(2g-2) x$ in $CH^1(C)_\qit$. Combining with Theorem \ref{main},
we obtain
\begin{Thm} \label{pure}
Let $C$ be a smooth projective curve of genus $g$ which admits a
pure ramification of degree $d$, then the tautological ring $\T$
on its Jacobian is generated by the classes $(p_i)_{i \leq d-1}$
and $q_1$.
\end{Thm}

For example, let $f_d(x, y)$ be a homogeneous polynomial of degree
$d$ such that the plane curve $C \ : \ (z^d + f_d(x,y)=0)$ is
smooth. The map $C \to \pit^1$ given by $[x:y:z] \mapsto [x:y]$
gives a pure ramification of $C$, thus Theorem \ref{pure} applies.

\section{The case of a hyperellitpic curve}

When $C$ is a hyperelliptic curve, we can obtain a more precise
description of $\T$ as follows:
\begin{Prop}\label{hyper}
Let $C$ be a hyperelliptic curve of genus $g$ and $a$ the smallest
integer such that $q_1^a =0$, then we have $$ \T= \frac{\qit
[p_1,q_1]}{( p_1^{g+1}, q_1 p_1^{g},\ldots, q_1^{a-1}
p_1^{g-a+2},q_1^{a}  )} .$$

\end{Prop}
Note that according to the proof of Proposition \ref{prop2} the integer $a$ is the dimension of $\T \cap CH^g(J)$.
\begin{proof}
By the precedent Theorem, $\T$ is generated by $p_1$ and $q_1$.
We present another direct proof of this as follows. We will use Proposition \ref{lem1} and prove
 that $p_2 \in \I_q$. Note $\mu$ for the hyperelliptic involution, and
 $D$ for the divisor $2x_0-(p+\mu(p))$ for any $p \in C$.
Then $-C$ is the translated of $C$ by $D$ in the Jacobian.
In terms of cycles we get $(-1)_* [C] = [D] * [C]$.
Let us consider $x_1$ a ramification point of $C$ and $\beta$ the class of the divisor
$x_1 - x_0$.
We get $D=(-2)_* \beta$. But we also have $\iota_* K=2(g-1)\beta$.
Recall that $\iota_* K = 2(\eta -[0])$ and apply $\F$ to get :
\begin{equation}\label{relation}
(-1)_* (p_1 + p_2 + \ldots + p_g) \ = \
\frac{1}{g-1} (-2)^* ( (g-1)[J]+q_1+q_2 + \dots ) * (p_1 + p_2 + \ldots + p_g).
\end{equation}
The relation we obtain in $CH^{2}_{(1)}(J)$ is $p_2=\frac{1}{g-1} p_1q_1$, so
$p_2 \in \I_q$.

We
claim that for any integers $(m,n)$ such that $m <a$ and $m+n <
g+1$, we have $q_1^m p_1^n \neq 0$. In fact, otherwise we can use
an induction and the following formula $$ \D (q_1^m p_1^n) =
n(m+n-g-1) q_1^m p_1^{n-1}$$ to deduce that $q_1^m = 0$, which is
a contradiction.
\end{proof}
\noindent \textbf{Remark :}
The relation $p_2=\frac{1}{g-1} p_1q_1$ is coherent  with the relation $p_2=p_1q_1$ given
as an exemple of Prop 4.2 in \cite{P2} for (hyperelliptic) curves of genus $2$.
Note that equation \eqref{relation}  also gives the relations
 $p_a = \sum_{i=1}^{a-1} \frac{(-2)^{i-1}}{g-1} p_{a-i}q_i$
when $a$ is even, and
$\sum_{i=1}^{a-1} \frac{(-2)^{i}}{g-1} p_{a-i}q_i=0$ when $a$ is odd. \\ \\
If we choose the base point $x_0$ to be one of the ramification
points, then $q_1=0$ and  $\T =\frac{\qit [p_1]}{(p_1^{g+1})},$ so
in this case, Poincar\'e's formula holds modulo rational
equivalence, which has been firstly proved in \cite{Co} (Theorem
2). \\ \\

\quad \\
C.N.R.S., Labo. J. Leray, Facult\'e des sciences, Universit\'e de NANTES \\
2, Rue de la Houssini\`ere,  BP 92208,
             F-44322 Nantes Cedex 03  France \\
fu@math.univ-nantes.fr  \\
\quad \\
IUFM NICE / Laboratoire GRIM - Universit\'e de Toulon et du Var\\
Boulevard Toussaint Merle, 83500 La Seyne Sur Mer, France \\
herbaut@math.unice.fr

\end{document}